\documentclass[numbook,envcountsame,smallcondensed,draft]{amsart}

\usepackage{cite,amsmath,amssymb,amsthm,latexsym,a4,graphicx,gastex}

\newtheorem{theorem}{Theorem}[section]
\newtheorem{proposition}[theorem]{Proposition}
\newtheorem{lemma}[theorem]{Lemma}
\newtheorem{corollary}[theorem]{Corollary}
\newtheorem{remark}[theorem]{Remark}

\DeclareMathOperator{\con}{con}
\DeclareMathOperator{\Con}{Con}
\DeclareMathOperator{\Eq}{Eq}
\DeclareMathOperator{\GCon}{GCon}
\DeclareMathOperator{\Orb}{Orb}
\DeclareMathOperator{\partition}{part}
\DeclareMathOperator{\Stab}{Stab}
\DeclareMathOperator{\Sub}{Sub}
\DeclareMathOperator{\var}{var}

\makeatletter

\renewcommand*\subjclass[2][2010]{\def\@subjclass{#2}\@ifundefined{subjclassname@#1}{\ClassWarning{\@classname}{Unknown edition (#1) of Mathematics Subject Classification; using '2010'.}}{\@xp\let\@xp\subjclassname\csname subjclassname@#1\endcsname}}

\makeatother

\begin{document}

\title[Cancellable elements in the lattice of varieties]{Cancellable elements in the lattice of overcommutative semigroup varieties}
\thanks{Both the authors are supported by the Ministry of Science and Higher Education of the Russian Federation (project 1.6018.2017/8.9) and by Russian Foundation for Basic Research (the first author by grant 18-31-00443; the second author by grant 17-01-00551).}

\author[V. Yu. Shaprynski\v{\i}]{Vyacheslav Yu. Shaprynski\v{\i}}
\address{Institute of Natural Sciences and Mathematics, Ural Federal University, Lenina str. 51, 620000 Ekaterinburg, Russia}
\email{vshapr@yandex.ru}

\author[B. M. Vernikov]{Boris M. Vernikov}
\address{Institute of Natural Sciences and Mathematics, Ural Federal University, Lenina str. 51, 620000 Ekaterinburg, Russia}
\email{bvernikov@gmail.com}



\begin{abstract}
We completely determine all cancellable elements in the lattice $\mathbb{OC}$ of overcommutative semigroup varieties. In particular, we prove that an overcommutative semigroup variety is a cancellable element of the lattice  $\mathbb{OC}$ if and only if it is a neutral element of this lattice.
\end{abstract}

\maketitle

\section{Introduction}
\label{intr}

The class of all semigroup varieties forms a lattice under the following naturally defined operations: for varieties $\mathbf X$ and $\mathbf Y$, their \emph{join} $\mathbf{X\vee Y}$ is the variety generated by the set-theoretical union of $\mathbf X$ and $\mathbf Y$ (as classes of semigroups), while their \emph{meet} $\mathbf{X\wedge Y}$ coincides with the set-theoretical intersection of $\mathbf X$ and $\mathbf Y$. This lattice has been intensively studied for more than five decades. A systematic overview of the material accumulated here is given in the survey \cite{Shevrin-Vernikov-Volkov-09}. We will denote the lattice of all semigroup varieties by $\mathbb{SEM}$. It is well known that the lattice $\mathbb{SEM}$ is the disjoint union of two large sublattices with essentially different properties: the coideal $\mathbb{OC}$ of all \emph{overcommutative} varieties (that is, varieties containing the variety of all commutative semigroups) and the ideal of all \emph{periodic} varieties (that is, varieties consisting of periodic semigroups). The global structure of the lattice $\mathbb{OC}$ has been revealed by Volkov in~\cite{Volkov-94}. It is proved there that this lattice is decomposed into a subdirect product of certain its intervals and each of these intervals is anti-isomorphic to the congruence lattice of a certain unary algebra of a special type (so-called $G$-set). We reproduce this result below (see Proposition~\ref{OC-struct}).

In the lattice theory, a significant attention is paid to the consideration of special elements of different types. Recall definitions of types of elements that will be mentioned below. An element $x$ of a lattice $\langle L;\ \vee,\wedge\rangle$ is called
\begin{align*}
&\text{\emph{cancellable} if}\quad&&\forall\,y,z\in L\colon\quad x\vee y=x\vee z\ \&\ x\wedge y=x\wedge z\longrightarrow y=z;\\
&\text{\emph{distributive} if}\quad&&\forall\,y,z\in L\colon\quad x\vee(y\wedge z)=(x\vee y)\wedge(x\vee z);\\
&\text{\emph{standard} if}\quad&&\forall\,y,z\in L\colon\quad(x\vee y)\wedge z=(x\wedge z)\vee(y\wedge z);\\
&\text{\emph{modular} if}\quad&&\forall\,y,z\in L\colon\quad y\le z\longrightarrow(x\vee y)\wedge z=(x\wedge z)\vee y;
\end{align*}
\emph{neutral} if, for all $y,z\in L$, the sublattice of $L$ generated by $x$, $y$ and $z$ is distributive. It is well known (see~\cite[Theorem~254]{Gratzer-11}, for instance) that an element $x\in L$ is neutral if and only if
$$
\forall\,y,z\in L\colon\quad(x\vee y)\wedge(y\vee z)\wedge(z\vee x)=(x\wedge y)\vee(y\wedge z)\vee(z\wedge x).
$$
\emph{Codistributive} and \emph{costandard} elements are defined dually to distributive and standard ones respectively. An extensive information about elements of all these types in abstract lattices may be found in~\cite[Section~III.2]{Gratzer-11}, for instance. Note that any neutral element is standard and costandard, any [co]standard element is both [co]distributive and cancellable, and any cancellable element is modular. All these claims are evident except the statement that a [co]standard element is [co]distributive; the verification of this fact may be found in~\cite[Theorem~253]{Gratzer-11}, for instance.{\sloppy

}

Over the past two decades, a number of papers have appeared devoted to the study of special elements of various types in the lattice $\mathbb{SEM}$ and some its sublattices, including the lattice $\mathbb{OC}$. An overview of results obtained in this area before 2015 may be found in the survey~\cite{Vernikov-15} (see also~\cite[Section~14]{Shevrin-Vernikov-Volkov-09}). From later works on this topic, we note the articles~\cite{Gusev-Skokov-Vernikov-18,Shaprynskii-Skokov-Vernikov-a,Skokov-Vernikov-19} devoted to examination of cancellable elements in the lattice $\mathbb{SEM}$. These elements are completely determined in~\cite{Shaprynskii-Skokov-Vernikov-a}, while earlier articles~\cite{Gusev-Skokov-Vernikov-18,Skokov-Vernikov-19} contain some partial results in this direction.

An examination of special elements in the lattice $\mathbb{OC}$ has been started by the second author in~\cite{Vernikov-01}. A description of five types of special elements (namely, distributive, codistributive, standard, costandard and neutral elements) in $\mathbb{OC}$ has been presented there. But the considerations in~\cite{Vernikov-01} contain a gap, and the main result of this article is incorrect. Namely, it was proved in~\cite{Vernikov-01} that, for an overcommutative semigroup variety, the properties of being a distributive, codistributive, standard, costandard or neutral element of $\mathbb{OC}$ are equivalent. This result of~\cite{Vernikov-01} is true. But, besides that, the main result of~\cite{Vernikov-01} contains a list of all overcommutative varieties that possess the five mentioned properties. This list is incomplete. All varieties from the list really have all the mentioned properties, but there are many other such varieties. A correct description of special elements of five mentioned types in $\mathbb{OC}$ is given in~\cite{Shaprynskii-Vernikov-11}.{\sloppy

}

The aim of this article is to classify all cancellable elements of the lattice $\mathbb{OC}$. In fact, we prove that an overcommutative semigroup variety is a cancellable element in $\mathbb{OC}$ if and only if it is a neutral, [co]standard or [co]distributive element in $\mathbb{OC}$.

The article is structured as follows. In Section~\ref{prel}, we introduce a necessary notation and formulate the main result of the article (Theorem~\ref{main}). In Section~\ref{G-set}, we recall a necessary information about $G$-sets. In Section~\ref{OC-structure}, we reproduce results of the article~\cite{Volkov-94}. Finally, Section~\ref{proof} is devoted to the proof of Theorem~\ref{main}. 

\section{Preliminaries and summary}
\label{prel}

We denote by $F$ the free semigroup over a countably infinite alphabet $A=\{x_1,x_2$, $\dots,x_n,\dots\}$. Elements of both $F$ and $A$ are denoted by small Latin letters. However, elements of $F$ for which it is not known exactly that they belong to $A$ are written in bold. As usual, elements of $A$ and $F$ are called \emph{letters} and \emph{words} respectively. We connect two sides of identities by the symbol $\approx$ and use the symbol $=$, among other things, for the equality relation on $F$. If $\mathbf u$ is a word then $\ell(\mathbf u)$ denotes the length of $\mathbf u$, $\ell_i(\mathbf u)$ is the number of occurrences of the letter $x_i$ in $\mathbf u$ and $\con(\mathbf u)$ stands for the set of all letters occurring in $\mathbf u$. An identity $\mathbf{u\approx v}$ is called \emph{balanced} if $\ell_i(\mathbf u)=\ell_i(\mathbf v)$ for all $i$. It is a common knowledge that if an overcommutative semigroup variety satisfies some identity then this identity is balanced.

Let $m$ and $n$ be integers with $2\le m\le n$. A \emph{partition of the number $n$ into $m$ parts} is a sequence of positive integers $\lambda=(\ell_1,\ell_2,\dots,\ell_m)$ such that
$$
\ell_1\ge\ell_2\ge\cdots\ge\ell_m\quad\text{and}\quad\sum_{i=1}^m\ell_i=n.
$$
We denote by $\Lambda_{n,m}$ the set of all partitions of the number $n$ into $m$ parts and put $\Lambda=\bigcup_{2\le m\le n}\Lambda_{n,m}$.

If $\mathbf u$ is a word then we denote by $\partition(\mathbf u)$ the partition of the number $\ell(\mathbf u)$ into $|\con(\mathbf u)|$ parts consisting of integers $\ell_i(\mathbf u)$ for all $i$ such that $x_i\in\con(\mathbf u)$ (the numbers $\ell_i(\mathbf u)$ are placed in $\partition(\mathbf u)$ in non-increasing order). If $\mathbf{u\approx v}$ is a balanced identity then, obviously, $\ell(\mathbf u)=\ell(\mathbf v)$, $|\con(\mathbf u)|=|\con(\mathbf v)|$ and $\partition(\mathbf u)=\partition(\mathbf v)$. We call the number $\ell(\mathbf u)=\ell(\mathbf v)$ a \emph{length} of the balanced identity $\mathbf{u\approx v}$.

Let $\lambda=(\ell_1,\ell_2,\dots,\ell_m)\in\Lambda_{n,m}$. We denote by $W_{n,m,\lambda}$, or simply $W_\lambda$, the set of all words $\mathbf u$ such that $\ell(\mathbf u)=n$, $\con(\mathbf u)=\{x_1,x_2,\dots,x_m\}$, $\ell_i(\mathbf u)\ge\ell_{i+1}(\mathbf u)$ for all $i=1,2,\dots,m-1$ and $\partition(\mathbf u)=\lambda$. It is evident that every balanced identity $\mathbf{u\approx v}$ with $\ell(\mathbf u)=\ell(\mathbf v)=n$, $|\con(\mathbf u)|=|\con(\mathbf v)|=m$ and $\partition(\mathbf u)=\partition(\mathbf v)=\lambda$ is equivalent to some identity $\mathbf{s\approx t}$ with $\mathbf s,\mathbf t\in W_{n,m,\lambda}$.

We call sets of the kind $W_{n,m,\lambda}$ \emph{transversals}. We say that an overcommutative variety $\mathbf V$ \emph{reduces} [\emph{collapses}] a transversal $W_{n,m,\lambda}$ if $\mathbf V$ satisfies some non-trivial identity [all identities] of the kind $\mathbf{u\approx v}$ with $\mathbf u,\mathbf v\in W_{n,m,\lambda}$. An overcommutative variety $\mathbf V$ is said to be \emph{greedy} if it collapses any transversal it reduces. The following assertion readily follows from the proof of~\cite[Theorem~2]{Vernikov-01} (and the corresponding part of the proof in~\cite{Vernikov-01} is correct).

\begin{proposition}
\label{neutral etc are greedy}
An overcommutative semigroup variety is a neutral \textup[standard, costandard, distributive, codistributive\textup] element of the lattice $\mathbb{OC}$ if and only if it is greedy.\qed
\end{proposition}

For a partition $\lambda=(\ell_1,\ell_2,\dots,\ell_m)\in\Lambda_{n,m}$, we define numbers $q(\lambda)$, $r(\lambda)$, and $s(\lambda)$ by the following way:
\begin{align*}
&q(\lambda)\ \text{is the number of}\ \ell_i\text{'s with}\ \ell_i=1\ \text{(if}\ \ell_m>1\ \text{then}\ q(\lambda)=0\text{);}\\
&r(\lambda)\ \text{is the sum of all}\ \ell_i\text{'s with}\ \ell_i>1\ \text{(if}\ \ell_1=1\ \text{then}\ r(\lambda)=0\text{);}\\
&s(\lambda)=\max\,\{r(\lambda)-q(\lambda)-\delta,0\}
\end{align*}
where
$$\delta=
\begin{cases}
0&\text{whenever}\ n=3,m=2,\ \text{and}\ \lambda=(2,1),\\
1&\text{otherwise}.
\end{cases}$$
If $k$ is a non-negative integer then $\lambda^k$ stands for the following partition of the number $n+k$ into $m+k$ parts:
$$
\lambda^k=(\ell_1,\ell_2,\dots,\ell_m,\underbrace{1,\dots,1}_{k\ \text{times}})
$$
(in particular, $\lambda^0=\lambda$).

We denote by $\var\Sigma$ the semigroup variety given by the identity system $\Sigma$. For a partition $\lambda\in\Lambda_{n,m}$, we put
$$
\mathbf W_{n,m,\lambda}=\var\{\mathbf{u\approx v}\mid \mathbf u,\mathbf v\in W_{n,m,\lambda}\}\quad\text{and}\quad\mathbf S_\lambda=\bigwedge_{i=0}^{s(\lambda)}\mathbf W_{n+i,m+i,\lambda^i}.
$$
We denote by $\mathbf{SEM}$ the variety of all semigroups. The main result of this article is the following

\begin{theorem}
\label{main}
For an overcommutative semigroup variety $\mathbf V$, the following are equivalent:
\begin{itemize}
\item[\textup{a)}] $\mathbf V$ is a cancellable element of the lattice $\mathbb{OC}$;
\item[\textup{b)}] $\mathbf V$ is a greedy variety;
\item[\textup{c)}] either $\mathbf{V=SEM}$ or $\mathbf V=\bigwedge_{i=1}^k\mathbf S_{\lambda_i}$ for some $\lambda_1,\lambda_2,\dots,\lambda_k\in\Lambda$.
\end{itemize}
\end{theorem}

Note that the equivalence of the claims~b) and~c) of this theorem is proved in~\cite[Proposition~2.4]{Shaprynskii-Vernikov-11}.

Theorem~\ref{main} together with~\cite[Theorem~2.2]{Shaprynskii-Vernikov-11} (see also Proposition~\ref{neutral etc are greedy} above) immediately imply the following

\begin{corollary}
\label{6 types of elements}
For an overcommutative semigroup variety $\mathbf V$, the following are equivalent:
\begin{itemize}
\item[\textup{a)}] $\mathbf V$ is a neutral element of the lattice $\mathbb{OC}$;
\item[\textup{b)}] $\mathbf V$ is a standard element of the lattice $\mathbb{OC}$;
\item[\textup{c)}] $\mathbf V$ is a costandard element of the lattice $\mathbb{OC}$;
\item[\textup{d)}] $\mathbf V$ is a distributive element of the lattice $\mathbb{OC}$;
\item[\textup{e)}] $\mathbf V$ is a codistributive element of the lattice $\mathbb{OC}$;
\item[\textup{f)}] $\mathbf V$ is a cancellable element of the lattice $\mathbb{OC}$.\qed
\end{itemize}
\end{corollary}
 
\section{$G$-sets}
\label{G-set}

Let $G$ be a group that acts on a set $A$. If $g\in G$ and $x\in A$ then we denote by $g(x)$ the image of $x$ under the action of $g$. An algebra with the carrier $A$ and the set of unary operations $G$ is called a $G$-\emph{set}. A preliminary information on $G$-sets and, in particular, on their congruences, may be found in the monograph~\cite{McKenzie-McNulty-Taylor-87}. 

A $G$-set $A$ is said to be \emph{transitive} if, for any two elements $x,y\in A$, there is an element $g\in G$ such that $g(x)=y$. A transitive $G$-subset of a $G$-set $A$ is called an \emph{orbit} of $A$. Clearly, any $G$-set is a disjoint union of all its orbits. The set of all orbits of a $G$-set $A$ is denoted by $\Orb(A)$. As usual, the congruence lattice on $A$ is denoted by $\Con(A)$, and the equivalence lattice on the set $X$ by $\Eq(X)$. 

Let $\alpha\in\Con(A)$ and $B$ and $C$ be distinct orbits in $A$. We say that $\alpha$ \emph{isolates} $B$ if $x\in B$ and $x\,\alpha\,y$ imply $y\in B$; $\alpha$ \emph{connects} $B$ and $C$ if $x\,\alpha\,y$ for some $x\in B$ and $y\in C$; $\alpha$ \emph{collapses} orbits $B$ and $C$ [an orbit $B$] if $x\,\alpha\,y$ for all $x,y\in B\cup C$ [respectively, all $x,y\in B$]. A congruence $\alpha$ is said to be \emph{greedy} if it collapses any pair of orbits it connects. Denote by $\GCon(A)$ the set of all greedy congruences of a $G$-set $A$. In~\cite{Vernikov-97}, it is shown that $\GCon(A)$ is a sublattice of $\Con(A)$ and the structure of this sublattice is characterized. Let us formulate these results in the form used below. For any congruence $\alpha$ on $A$ we introduce a binary relation $\alpha^\star$ on $\Orb(A)$ by the following rule: if $B,C\in\Orb(A)$, then $B\,\alpha^\star\,C$ if and only if either $B=C$ or $\alpha$ connects $B$ and $C$. Obviously, $\alpha^\star$ is an equivalence relation on $\Orb(A)$.

\begin{lemma}[\mdseries{\cite[Lemma~1.1 and Proposition~1.2]{Vernikov-97}}]
\label{GCon}
Let $A$ be a $G$-set and $\Orb(A)=\{A_i\mid i\in I\}$. The set $\GCon(A)$ is a sublattice of the lattice $\Con(A)$. The map $f$ from $\GCon(A)$ into $\Eq(\Orb(A))\times\prod_{i\in I}\Con(A_i)$ given by the rule
$$
f(\alpha)=(\alpha^\star;\dots,\alpha_i,\dots)
$$
where $\alpha_i$ is the restriction of the congruence $\alpha\in\GCon(A)$ to the orbit $A_i$ is an isomorphic embedding.\qed
\end{lemma}

If $A$ is a $G$-set and $a\in A$ then we put
$$
\Stab_A(a)=\{g\in G\mid g(a)=a\}.
$$
It is clear that $\Stab_A(a)$ is a subgroup of $G$. It is called the \emph{stabilizer} of the element $a$ in $A$. Let $B$ and $C$ be two distinct orbits of a $G$-set $A$, $b\in B$ and $c\in C$. We denote by $\rho_{b,c}$ the binary relation on $A$ given by the following rule: $x\,\rho_{b,c}\,y$ if and only if either $x=y$ or $\{x,y\}=\{g(b),g(c)\}$ for some $g\in G$.

\begin{lemma}[\mdseries{\cite[Lemma~3]{Vernikov-01}}]
\label{rho is congruence}
If $\Stab_A(b)=\Stab_A(c)$ then $\rho_{b,c}$ is a congruence on $A$.\qed
\end{lemma}

The following assertion follows from the well-known group-theoretical fact (see~\cite[the claim~(1) of the statement~(5.9)]{Aschbacher-00}, for instance).

\begin{lemma}
\label{isomorphic orbits}
If $A$ is a non-transitive $G$-set and $\Stab_A(x)=\Stab_A(y)$ for any $x,y\in A$ then any two distinct orbits of $A$ are isomorphic.\qed
\end{lemma}

Note that lattices of equivalence relations are congruence lattices of some specific $G$-sets. Indeed, let $T=\{e\}$ be the singleton group and $S$ be a set. Then $S$ can be considered as a trivial $T$-set with the action of $T$ given by the rule $e(x)=x$ for any $x\in S$. Clearly, any equivalence relation on $S$ is the congruence of the $T$-set $S$, so the lattice $\Eq(S)$ is the congruence lattice of this $T$-set. If $S$ is a set then we denote by $\Delta_S$ the universal relation on $S$ and by $\nabla_S$ the equality relation on $S$. If $X$ is a non-empty subset of $S$ then we put $\rho_X=(X\times X)\cup\nabla_S$. Clearly, $\rho_X$ is an equivalence relation on $S$.

The following assertion plays the key role in the proof of Theorem~\ref{main}.

\begin{proposition}
\label{canc in G-set}
Let $A$ be a non-transitive $G$-set with $\Stab_A(x)=\Stab_A(y)$ for any $x,y\in A$. A congruence $\alpha$ on $A$ is a cancellable element of the lattice $\Con(A)$ if and only if $\alpha$ is either the universal relation or the equality relation on $A$.
\end{proposition}

\begin{proof}
\emph{Sufficiency} is evident. One can prove \emph{necessity}. We divide the proof into three parts.

\smallskip

1) Here we prove that the congruence $\alpha$ is greedy. Arguing by contradiction, suppose that $\alpha$ connects but not collapses orbits $B$ and $C$ of $A$. Then there are $b\in B$ and $c\in C$ with $b\,\alpha\,c$. Let us define binary relations $\beta$ and $\gamma$ on $A$ by the following way: $x\,\beta\,y$ if and only if one of the following holds:
\begin{itemize}
\item[a)] $x,y\in B$,
\item[b)] $x,y\in C$ and $x\,\alpha\,y$,
\item[c)] $x,y\notin B\cup C$ and $x\,\alpha\,y$;
\end{itemize}
$x\,\gamma\,y$ if and only if one of the following holds:
\begin{itemize}
\item[a)] $x,y\in B$ and $x\,\alpha\,y$,
\item[b)] $x,y\in C$,
\item[c)] $x,y\notin B\cup C$ and $x\,\alpha\,y$.
\end{itemize}
It is evident that $\beta$ and $\gamma$ are congruences on $A$. Suppose that $\alpha$ collapses $B$. Let $x\in B$ and $y\in C$. There is an element $g\in G$ with $y=g(c)$. Then $x\,\alpha\,g(b)\,\alpha\,g(c)=y$, whence $x\,\alpha\,y$. Furthermore, let $x,y\in C$. Then $x=g(c)$ and $y=h(c)$ for some $g,h\in G$. Therefore,
$$
x=g(c)\,\alpha\,g(b)\,\alpha\,h(b)\alpha\,h(c)=y,
$$
whence $x\,\alpha\,y$ again. We see that $\alpha$ collapses $B$ and $C$, contradicting the choice of $\alpha$. Therefore, $\alpha$ does not collapse $B$. This implies that $\beta|_B\ne\gamma|_B$, whence $\beta\ne\gamma$. Furthermore, it is evident that $\alpha\wedge\beta=\alpha\wedge\gamma=\delta$ where $\delta$ is the congruence on $A$ defined by the following way: $x\,\delta\,y$ if and only if one of the following holds:
\begin{itemize}
\item[a)] $x,y\in B$ and $x\,\alpha\,y$,
\item[b)] $x,y\in C$ and $x\,\alpha\,y$
\item[c)] $x,y\notin B\cup C$ and $x\,\alpha\,y$.
\end{itemize}

Now we aim to verify that $\alpha\vee\beta=\alpha\vee\gamma=\alpha\vee\rho_{B\cup C}$. Suppose that $x\in B\cup C$. If $x\in B$ then $b\,\beta\,x$ by the definition of $\beta$. If $x\in C$ then there is $g\in G$ such that $x=g(c)$, so $b\,\beta\,g(b)\,\alpha\,g(c)=x$. In any case, $(b,x)\in\alpha\vee\beta$. We see that $B\cup C$ is contained in the $(\alpha\vee\beta)$-class of the element $b$. Hence $\rho_{B\cup C}\subseteq\alpha\vee\beta$. Hence $\alpha\vee\rho_{B\cup C}\subseteq\alpha\vee\beta$. The inverse inclusion is obvious. Hence $\alpha\vee\beta=\alpha\vee\rho_{B\cup C}$. The equality $\alpha\vee\gamma=\alpha\vee\rho_{B\cup C}$ can be verified analogously.

We see that $\alpha\vee\beta=\alpha\vee\gamma$, $\alpha\wedge\beta=\alpha\wedge\gamma$ and $\beta\ne\gamma$, contradicting the claim that $\alpha$ is cancellable. Thus, we have proved that the congruence $\alpha$ is greedy.

\smallskip

2) Now we prove that if $\alpha\ne\Delta_A$ then $\alpha$ isolates each orbit of $A$. Indeed, we prove above that $\alpha\in\GCon(A)$. Let $\alpha^\star$ be the equivalence relation on the set $\Orb(A)$ defined before Lemma~\ref{GCon}. Since each component of a modular element in a subdirect product is modular, Lemma~\ref{GCon} implies that $\alpha^\star$ is a modular element of the lattice $\Eq(\Orb(A))$. According to \cite[Proposition~2.2]{Jezek-81}, this implies that $\alpha^\star=\rho_N$ for some subset $N$ of $\Orb(A)$. Suppose that $\alpha^\star$ differs from $\nabla_{\Orb(A)}$ and $\Delta_{\Orb(A)}$. Then $1<|N|<|\Orb(A)|$. Let $X,Y\in N$ and $X\ne Y$. Put $M=\Orb(A)\setminus N$. Consider the sets $B=\overline M\cup X$ and $C=\overline M\cup Y$ where $\overline M$ is the join of all orbits from $M$ and the congruences $\beta=\rho_B\vee\rho_Y$ and $\gamma=\rho_C\vee\rho_X$. It is evident that $\beta\ne\gamma$. 

Let us check that $\alpha\vee\beta=\alpha\vee\gamma=\Delta_A$. Fix an element $x\in X$. Note that $(x,z)\in\alpha\vee\beta$ for any $z\in A$. Indeed, if $z\in B$ then $(x,z)\in\rho_B$. If $z\not\in B$ then $z$ lies in some orbit $Z$ from $N$. Since $(X,Z)\in\rho_N=\alpha^\star$, we have $x'\,\alpha\, z'$ for some $x'\in X$ and $z'\in Z$. We have $x'=g(x)$ and $z'=h(z)$ for some $g,h\in G$. Hence
$$z=h^{-1}(z')\,\alpha\, h^{-1}(x')=h^{-1}g(x)\,\beta\, x.$$
We have checked that $\alpha\vee\beta=\Delta_A$. The equality $\alpha\vee\gamma=\Delta_A$ is analogous.

Now we will check that $\alpha\wedge\beta=\alpha\wedge\gamma$. Since $\alpha$ isolates each orbit from $M$ and $\beta$ isolates each orbit from $N\setminus X$, the congruence $\alpha\wedge\beta$ isolates each orbit from $M\cup(N\setminus X)$, that is, from $\Orb(A)\setminus X$. Since a congruence cannot isolate all orbits but one, the congruence $\alpha\wedge\beta$ isolates all orbits. Hence $(\alpha\wedge\beta)^\star=\nabla_{\Orb(A)}$. The same argument shows that $(\alpha\wedge\gamma)^\star=\nabla_{\Orb(A)}$. It is evident that the restriction of each of the congruences $\alpha\wedge\beta$ and $\alpha\wedge\gamma$ on any orbit from $M\cup X\cup Y$ coincides with the restriction of $\alpha$ on the same orbit. Furthermore, the restriction of $\alpha\wedge\beta$ or $\alpha\wedge\gamma$ on any orbit from $\Orb(A)\setminus(M\cup X\cup Y)$ is trivial. Hence, by Lemma~\ref{GCon}, we have $\alpha\wedge\beta=\alpha\wedge\gamma$.
This contradicts the cancellability of $\alpha$.

We have proved that $\alpha^\star=\nabla_{\Orb(A)}$ or $\alpha^\star=\Delta_{\Orb(A)}$. Since $\alpha$ is greedy, $\alpha^\star=\Delta_{\Orb(A)}$ implies that $\alpha=\Delta_A$, which is not the case. Therefore, $\alpha^\star=\nabla_{\Orb(A)}$. This means exactly that $\alpha$ isolates each orbit.

\smallskip 

3) Now we are ready to complete the proof. Let $\alpha\ne\Delta_A$. We need to check that $\alpha=\nabla_A$. In view of what has been checked in the previous paragraph, it suffices to verify that $\alpha|_B=\nabla_B$ for any orbit $B$ of $A$. Arguing by contradiction, suppose that $x\,\alpha\,y$ for some distinct elements $x,y\in B$. Let $C$ be an orbit of $A$ with $B\ne C$. According to Lemma~\ref{isomorphic orbits}, there is an isomorphism $\varphi$ from $B$ onto $C$. Put $\beta=\rho_{x,\varphi(x)}$ and $\gamma=\rho_{x,\varphi(y)}$. Lemma~\ref{rho is congruence} shows that $\beta$ and $\gamma$ are congruences on $A$. Clearly, the restriction of each of the congruences $\beta$ and $\gamma$ on any orbit of $A$ is the equality relation on this orbit. Since $\alpha$ isolates any orbit of $A$, this implies that $\alpha\wedge\beta=\alpha\wedge\gamma=\nabla_A$. 

Now we are going to prove that $\alpha\vee\beta=\alpha\vee\gamma$. Since $x,y\in B$, we have $y=g(x)$ for some $g\in G$. Furthermore, $x\,\alpha\, g(x)$ implies $g^{-1}(x)\,\alpha\, x$. Hence
$$x\,\alpha\,g^{-1}(x)\,\gamma\,g^{-1}(\varphi(y))=g^{-1}(\varphi(g(x)))=\varphi(x).$$
So we have $(x,\varphi(x))\in\alpha\vee\gamma$. Since the congruence $\rho_{x,\varphi(x)}$ is generated by the pair $(x,\varphi(x))$, this implies that $\beta\subseteq\alpha\vee\gamma$ whence $\alpha\vee\beta\subseteq\alpha\vee\gamma$. Furthermore,
$$x\,\alpha\,g(x)\,\beta\, g(\varphi(x))=\varphi(g(x))=\varphi(y).$$
So we have $(x,\varphi(y))\in\alpha\vee\beta$. Hence $\gamma\subseteq\alpha\vee\beta$ and $\alpha\vee\gamma\subseteq\alpha\vee\beta$. Hence $\alpha\vee\beta=\alpha\vee\gamma$.

We have verified that $\alpha\wedge\beta=\alpha\wedge\gamma$ and $\alpha\vee\beta=\alpha\vee\gamma$. It is evident that $\beta\ne\gamma$. This contradicts the choice of $\alpha$ as a cancellable element of $\Con(A)$.
\end{proof}

We denote by $S_n$ the full symmetric group on the set $\{1,2,\dots,n\}$. The subgroup lattice of a group $G$ is denoted by $\Sub(G)$. We need also the following

\begin{lemma}[\mdseries{\cite[Lemma~2.3]{Shaprynskii-Skokov-Vernikov-a}}]
\label{canc in S_n}
Let $n$ be a natural number. A subgroup $H$ of the group $S_n$ is a cancellable element of the lattice $\Sub(S_n)$ if and only if either $H=T$ or $H=S_n$.\qed
\end{lemma}

\section{The structure of the lattice $\mathbb{OC}$}
\label{OC-structure}

For a positive integer $n$ with $n\ge 2$, we denote by $\mathbf C_n$ the variety of semigroups defined by all balanced identities of length $\ge n$. It is clear that
$$
\mathbf{COM}=\mathbf C_2\subset\mathbf C_3\subset\cdots\subset\mathbf C_n\subset\cdots\subset\mathbf{SEM}
$$
where $\mathbf{COM}$ stands for the variety of all commutative semigroups. Further, let $m$ be a positive integer with $2\le m\le n$. Denote by $\mathbf C_{n,m}$ the variety of semigroups defined by all balanced identities of length $>n$ and all balanced identities of length $n$ depending on $\le m$ letters. For notational convenience, we put also $\mathbf C_{n,1}=\mathbf C_{n+1}$. It is clear that
$$
\mathbf C_n=\mathbf C_{n,n}\subset\mathbf C_{n,n-1}\subset\cdots\subset\mathbf C_{n,2}\subset\mathbf C_{n,1}=\mathbf C_{n+1}.
$$
Finally, let $\lambda=(\lambda_1,\lambda_2,\dots,\lambda_m)$ be a partition of the number $n$ into $m$ parts. Denote by $\mathbf C_\lambda$ the subvariety of the variety $\mathbf C_{n,m-1}$ which is defined in $\mathbf C_{n,m-1}$ by all balanced identities $\mathbf{u\approx v}$ of length $n$ such that $\con(\mathbf u)=\{x_1,x_2,\dots,x_m\}$ and $\partition(\mathbf u)=\lambda$. It is clear that
$$
\mathbf C_{n,m}\subset\mathbf C_\lambda\subset\mathbf C_{n,m-1}.
$$
Denote by $I_\lambda$ the interval $[\mathbf C_\lambda,\mathbf C_{n,m-1}]$ of the lattice $\mathbb{OC}$. 

For a partition $\lambda=(\lambda_1,\lambda_2,\dots,\lambda_m)\in\Lambda_{n,m}$, we put
$$
S_\lambda=\{\sigma\in S_m\mid\lambda_i=\lambda_{\sigma(i)}\ \text{for}\ i=1,2,\dots,m\}.
$$
It is clear that $S_\lambda$ is a subgroup of $S_m$. For any word $\mathbf u\in W_\lambda$ and any permutation $\sigma\in S_\lambda$, let $\sigma(\mathbf u)$ be the word obtained from $\mathbf u$ by replacing each occurrence of a letter $x_i$ by $x_{\sigma(i)}$ for all $i=1,2,\dots,m$. The definition of the group $S_\lambda$ implies that $\sigma(\mathbf u)\in W_\lambda$. Obviously, $W_\lambda$ is an $S_\lambda$-set relatively to the just defined action of the group $S_\lambda$. 

\begin{proposition}[\!\!\mdseries{\cite[Propositions~2.2 and~3.1 and Theorem 4.1]{Volkov-94}}]
\label{OC-struct}
The following are true:
\begin{itemize}
\item[\textup{(i)}] the lattice of all overcommutative semigroup varieties is decomposed into a subdirect product of intervals of the form $I_\lambda$ where $\lambda$ runs over the set $\Lambda$;
\item[\textup{(ii)}] for any $\lambda\in\Lambda$, the interval $I_\lambda$ is anti-isomorphic to the congruence lattice of the $S_\lambda$-set $W_\lambda$.\qed
\end{itemize}
\end{proposition}

We will denote by $\mathbf V_\lambda$ the image of a variety $\mathbf V$ in an interval $I_\lambda$ defined by Proposition~\ref{OC-struct}.

If $\mathbf u\in W_\lambda$ and $\sigma$ is a non-trivial permutation from $S_\lambda$ then $\sigma(\mathbf u)\ne\mathbf u$. This implies the following observation formulated for convenience of references.

\begin{remark}
\label{stab are trivial and coincide}
$\Stab_{W_\lambda}(\mathbf u)=T$ for each $\mathbf u\in W_\lambda$; therefore, $\Stab_{W_\lambda}(\mathbf u)=\Stab_{W_\lambda}(\mathbf v)$ for any $\mathbf u,\mathbf v\in W_\lambda$.\qed
\end{remark}

\section{Proof of Theorem \ref{main}}
\label{proof}

Here we aim to verify Theorem~\ref{main}. The equivalence~b)\,$\leftrightarrow$\,c) is verified in~\cite[Proposition~2.4]{Shaprynskii-Vernikov-11}. The implication~b)\,$\rightarrow$\,a) follows from Proposition~\ref{neutral etc are greedy} and the fact that a neutral element of a lattice is cancellable. It remains to check the implication~a)\,$\rightarrow$\,b). To achieve this goal, we use the same arguments as in the proof of~\cite[Theorem~2]{Vernikov-01}. For reader convenience and in the sake of completeness, we reproduce these arguments here without references to~\cite{Vernikov-01}.

Let $\mathbf V$ be an overcommutative variety of semigroups which is a cancellable element of the lattice $\mathbb{OC}$. We need to verify that $\mathbf V$ is greedy. Denote by $\nu$ the fully invariant congruence on a semigroup $F$ corresponding to the variety $\mathbf V$. It is clear that the congruence $\nu$ is subcommutative, i.e., it is contained in the fully invariant congruence on $F$ that corresponds to the variety $\mathbf{COM}$. Denote by $\mathbb{SC}$ the lattice of all subcommutative fully invariant congruences on $F$. It is clear that this lattice is anti-isomorphic to the lattice $\mathbb{OC}$. Now Proposition~\ref{OC-struct} implies that the lattice $\mathbb{SC}$ is isomorphic to the subdirect product of lattices of the form $\Con(W_\lambda)$ where $\lambda$ runs over $\Lambda$. The proof of this result given in~\cite{Volkov-94} shows that the projection to $\Con(W_\lambda)$ of the image of the congruence $\nu$ under the isomorphic embedding $\mathbb{SC}$ in $\prod_{\lambda\in\Lambda}\Con(W_\lambda)$ is simply the restriction of the congruence $\nu$ to $W_\lambda$. Denote this restriction by $\nu_\lambda$. 

The statement we are to prove is obviously equivalent to the claim that, for any $\lambda\in\Lambda$, the congruence $\nu_\lambda$ is either the universal relation or the equality relation on $W_\lambda$.
Consider any elements $\mathbf U,\mathbf W\in I_\lambda$ with $\mathbf V_\lambda\vee\mathbf U=\mathbf V_\lambda\vee\mathbf W$ and $\mathbf V_\lambda\wedge\mathbf U=\mathbf V_\lambda\wedge\mathbf W$. It directly follows from the definition of $I_\lambda$ that $\mathbf U_\lambda=\mathbf U$, $\mathbf W_\lambda=\mathbf W$ and $\mathbf U_\mu=\mathbf W_\mu$ for any $\mu\in\Lambda\setminus\{\lambda\}$. Hence $\mathbf V_\mu\vee\mathbf U_\mu=\mathbf V_\mu\vee\mathbf W_\mu$ and $\mathbf V_\mu\wedge\mathbf U_\mu=\mathbf V_\mu\wedge\mathbf W_\mu$ for each $\mu\in\Lambda$. Hence, by Proposition~\ref{OC-struct}, $\mathbf V\vee\mathbf U=\mathbf V\vee\mathbf W$ and $\mathbf V\wedge\mathbf U=\mathbf V\wedge\mathbf W$. Since $\mathbf V$ is cancellable, we have $\mathbf U=\mathbf W$. We have proved that $\mathbf V_\lambda$ is a cancellable element of $I_\lambda$. Therefore, the congruence $\nu_\lambda$ is a cancellable element of the lattice $\Con(W_\lambda)$ by Proposition~\ref{OC-struct}. Let $\lambda=(\lambda_1,\lambda_2,\dots,\lambda_m)$. Our further considerations are divided into two cases.

\smallskip

\emph{Case}~1: $\lambda_1>1$. In this case, $W_\lambda$ contains (among others) the words $\mathbf u=x^{\lambda_1}x^{\lambda_2}\cdots x^{\lambda_m}$ and $\mathbf v=x^{\lambda_1-1}x^{\lambda_2}\cdots x^{\lambda_m}x_1$. It is clear that $\sigma(\mathbf u)\ne\mathbf v$ for any permutation $\sigma\in S_\lambda$. Hence the $S_\lambda$-set $W_\lambda$ is non-transitive. According to Remark \ref{stab are trivial and coincide}, the stabilizers of any two elements of this $S_\lambda$-set coincide. Now Proposition~\ref{canc in G-set} applies with the desirable conclusion.

\smallskip

\emph{Case}~2: $\lambda_1=1$. Obviously, in this case $\lambda_2=\cdots=\lambda_m=1$, $S_\lambda=S_m$ and $W_\lambda$ is a transitive $S_m$-set. As is well known (see, e.g.,~\cite[Lemma 4.20]{McKenzie-McNulty-Taylor-87}), the congruence lattice of a transitive $G$-set $A$ is isomorphic to the interval $[\Stab_A(a),G]$ in the lattice $\Sub(G)$ where $a$ is an arbitrary element of $A$. According to Remark~\ref{stab are trivial and coincide}, the stabilizer of any element of the $S_\lambda$-set $W_\lambda$ is the singleton group. Hence $\Con(W_\lambda)$ is the whole lattice $\Sub(S_m)$. It remains to refer to Lemma~\ref{canc in S_n}.

\smallskip

Theorem \ref{main} is proved.\qed

\end{document}